\documentclass[a4paper, 10pt, twoside, notitlepage]{amsart}

\usepackage[utf8]{inputenc}
\usepackage{color}
\usepackage{amsmath} 
\usepackage{amssymb} 
\usepackage{amsthm}
\usepackage{geometry}
\usepackage{graphicx}
\usepackage{mathtools}
\usepackage{esint}
\usepackage[colorlinks=true,linkcolor=blue]{hyperref}

\theoremstyle{plain}
\newtheorem{thm}{Theorem}
\newtheorem{prop}{Proposition}[section]
\newtheorem{lem}[prop]{Lemma}
\newtheorem{cor}[prop]{Corollary}

\newtheorem{rmk}[prop]{Remark}

\newtheorem{example}[prop]{Example}

\newcommand {\R} {\mathbb{R}} 
 \newcommand {\N} {\mathbb{N}}

\newcommand {\D} {\Delta}

\newcommand {\supp} {\text{supp}}

\DeclareMathOperator{\spa} {span}

\begin{document}

\title[The Finite Dimensional Fractional Calder\'on Problem]{Lipschitz stability for the Finite Dimensional Fractional Calder\'on Problem with Finite Cauchy Data}

\author{Angkana Rüland}
\address{Max-Planck Institute for Mathematics in the Sciences, Inselstraße 22, 04103 Leipzig, Germany}
\email{rueland@mis.mpg.de}

\author{Eva Sincich}
\address{Dipartimento di Matematica e Geoscienze
Università degli Studi di Trieste,
via Valerio 12/1, 34127 Trieste
Italy}
\email{esincich@units.it}

\begin{abstract}
In this note we discuss the conditional stability issue  for the finite dimensional Calder\'on problem for the fractional Schr\"{o}dinger equation  with a finite number of measurements. More precisely, we assume that the unknown potential $q \in L^{\infty}(\Omega) $ in the equation $((-\D)^s+ q)u = 0 \mbox{ in } \Omega\subset \mathbb{R}^n$  satisfies the a priori assumption that it is contained in a finite dimensional subspace of $L^{\infty}(\Omega)$. Under this condition we prove Lipschitz stability estimates for the fractional Calder\'on problem by means of finitely many Cauchy data depending on $q$. We allow for the possibility of zero being a Dirichlet eigenvalue of the associated fractional Schrödinger equation. Our result relies on the strong Runge approximation property of the fractional Schrödinger equation.
\end{abstract}

\maketitle

\section{Introduction}

In this note we seek to prove Lipschitz stability estimates for the fractional Calder\'on problem with a finite number of measurements under structural a priori information on the potential. More precisely, we consider the following 
direct problem 
\begin{align}
\label{eq:frac_Cal}
\begin{split}
((-\D)^s+ q)u &= 0 \mbox{ in } \Omega, \\
u& = f \mbox{ in } \Omega_e,
\end{split}
\end{align}
where $\Omega \subset \R^n$ with $n\geq 1$ is an open, bounded domain and $\Omega_e=\R^n \setminus \overline{\Omega}$. 
Assuming that zero is not a Dirichlet eigenvalue, i.e., that the following condition holds
\begin{align}
\label{eq:not_zero}
[u \in H^{s}(\R^n), \ ((-\D)^s + q) u = 0 \mbox{ in } \Omega, \ u = 0 \mbox{ in } \Omega_e] \Rightarrow u \equiv 0 \mbox{ in } \R^n,
\end{align}
it is possible to formally define the (generalized) Dirichlet-to-Neumann map (c.f. \eqref{eq:DtN} for a weak formulation of it)
\begin{align}
\label{eq:DtN}
\Lambda_q: H^{s}(\Omega_e) \rightarrow \widetilde{H}^{-s}(\Omega_e), \
f \mapsto (-\D)^s u|_{\Omega_e}.
\end{align}
In \cite{GSU16} it was shown that the associated inverse problem -- the fractional Calder\'on problem -- in which one seeks to recover the potential $q$ in a suitable function space (e.g. $q \in L^{\infty}(\Omega)$) from the knowledge of the Dirichlet-to-Neumann map $\Lambda_q$, is injective. This uniqueness result was extended to (almost) scaling critical spaces, $q\in Z^{-s}(\Omega)$, in \cite{RS17}. Motivated by the formal well-definedness in terms of degrees of freedom versus unknowns, in \cite{GRSU18} it was proved that a single measurement suffices to reconstruct the potential for the fractional Calder\'on problem.
In \cite{RS17} it was also shown that for the full infinite dimensional problem with infinite measurements conditional stability with a logarithmic modulus holds. By virtue of the constructions in \cite{RS17d}, this dependence is optimal. We refer to \cite{CL18, GLX17, HL17, LL17, RS17, S17} and the references therein for further developments on the fractional Calder\'on and related problems.

In the present note, we show that as in the case of the \emph{classical} Calder\'on problem, the a priori information of the finite dimensionality of the potential $q$ allows one to improve the modulus of continuity in the stability estimate to a \emph{Lipschitz continuous} dependence. Moreover, as in the classical setting the Lipschitz constant however diverges exponentially with the increasing ``complexity" of the setting (c.f. Section \ref{sec:opti} for a more precise explanation of this). First results of this kind were proven in \cite{AV05, Ron06} for the classical Calder\'on problem, where a Lipschitz stability estimate and the exponential behavior of the Lipschitz constant were achieved for piecewise constant conductivities. This line of research was extended to many different directions (c.f. \cite{BF11, AdHGS16, GS15, AdHGS17, BdHQ13, AS18}). In \cite{AdHGS17} a conditional Lipschitz stability estimate for a piecewise linear potential in the classical Schr\"{o}dinger equation by means of local Cauchy data was achieved. To some extent our main results represent the nonlocal analogue of the latter as we also deal with Cauchy data which allow us to remove spectral conditions. We also prove that a finite number of measurements $\tilde{f}_j^{(k)}$ (see Theorem \ref{thm:Lip_frac_Runge} and Proposition \ref{prop:finite_Cauchy} for a precise definition), suffice to recover the potential in the finite dimensional case in a Lipschitz stable way.

Originally, these improved stability properties were derived by using two main ingredients (c.f. \cite{AV05}):
\begin{itemize}
\item Propagation of smallness estimates; 
\item Asymptotic estimates at the discontinuity interfaces for singular solutions.
\end{itemize} 
The first ingredient allows one to propagate the information from the boundary into the interior of the domain. It is based on quantitative unique continuation results as for instance in \cite{ARV09}. In order to combine this with Alessandrini's identity which yields control on the potentials/conductivities, the second ingredient, namely singular solutions, is used. The singular solutions are related to localized solutions in the sense of for instance \cite{G08, HPS17} and provide very concentrated information around the domains of interest and in particular near the discontinuity interfaces. 

In the setting of the \emph{fractional} Calder\'on problem the argument simplifies considerably. Due to the strong approximation properties which are valid in this setting (and which are dual to the strong uniqueness properties for the fractional Laplacian) \cite{GSU16, DSV14}, we do not need to construct singular solutions in this set-up. It suffices to use quantitative Runge approximation properties (c.f. Theorems \ref{thm:Runge_frac} and \ref{thm:quant_Runge_frac}, which were first deduced in \cite{RS17}). This is in the same spirit as the results in \cite{GSU16} and \cite{RS17}, where it was not necessary to construct CGO solutions but where these were replaced by the very strong Runge approximation properties that hold in the fractional set-up. Exploiting the strong Runge approximation property for the fractional Laplacian, our argument is of a ``perturbative flavour", which is reminiscent 
of the very recent work \cite{AS18} on the classical Calder\'on problem.

In the sequel we discuss the Runge approximation approach in the setting in which the potential $q$ will be assumed to be contained in a finite dimensional subspace of $L^{\infty}(\Omega)$. With only small modifications and by using the results from \cite{RS17}, it is possible to extend this to critical function spaces such as $L^{\frac{n}{2s}}(\Omega)$. We address both the situation in which zero is not a Dirichlet eigenvalue of the problem \eqref{eq:frac_Cal} (c.f. Theorem \ref{thm:Lip_frac_Runge}) and the setting in which zero is allowed to be a Dirichlet eigenvalue (c.f. Theorem \ref{thm:Lip_frac_Runge_eigen}). We derive Runge approximation properties in the latter, which might be interesting in their own right.

\subsection{Organization of the article}
After recalling some auxiliary properties from \cite{GSU16} in Section \ref{sec:pre}, we first consider the case in which the potential $q\in L^{\infty}(\Omega)$ is such that zero is \emph{not} a Dirichlet eigenvalue. In this case the main result on the improved stability estimates (c.f. Theorem \ref{thm:Lip_frac_Runge} and Corollary \ref{thm:Lip_frac_Runge_fulldata}) is a quite direct consequence of the quantitative Runge approximation properties which were deduced in \cite{RS17}. It is stated and proved in Section \ref{sec:non_zero}.

In Section \ref{sec:zero} we extend Theorem \ref{thm:Lip_frac_Runge} to the situation, in which $q\in L^{\infty}(\Omega)$ is such that it \emph{may} lead to a zero Dirichlet eigenvalue for the fractional Schrödinger operator \eqref{eq:frac_Cal}. In order to prove Theorem \ref{thm:Lip_frac_Runge_eigen} and Proposition \ref{prop:finite_Cauchy}, the main results in this set-up, we first extend the Runge approximation properties for fractional Schrödinger operators without zero eigenvalues to Schrödinger operators with potentials which allow for zero Dirichlet eigenvalues. As these are still Fredholm operators of index zero, we only have to deal with a finite dimensional additional subspace, which can be treated ``by hand", c.f. Sections \ref{sec:qual_Runge_zero}, \ref{sec:quant_Runge_zero}. In particular this allows us to argue similarly as in the case without zero eigenvalue, c.f. Section \ref{sec:Proof_zero}. 

Finally, in Section \ref{sec:opti}, we show that as in the classical case, in general the constants in the Lipschitz estimates depend exponentially on the ``complexity" of the problem, in that for the piecewise constant setting the constant depends exponentially on the number of underlying domains. This is based on a combination of ideas from \cite{Ron06} and \cite{RS17d}.

\section{Preliminaries}
\label{sec:pre}

In this section, we briefly recall some auxiliary properties from \cite{GSU16}, which will be used throughout the article. We first define the relevant function spaces and then discuss the well-posedness of the forward problem and the weak definition of the fractional Dirichlet-to-Neumann map in the absence of zero eigenvalues.

\subsection{Function spaces}

In the sequel, we will use the following $L^2$ based fractional Sobolev spaces (c.f. Chapter 3 in \cite{McLean}). Let $W \subset \R^n$ be open. Then, for $s\in \R$ we set:
\begin{align*}
&H^{s}(W)= \{y : \mbox{ there exists } Y\in H^{s}(\R^n) \mbox{ such that } y= Y|_{W} \},\\ 
&\|y\|_{H^{s}(W)}:= \inf\{\|Y\|_{H^{s}(\R^n)}: Y|_{W}=y\},\\
&\widetilde{H}^{s}(W)= \mbox{ closure of } C_c^{\infty}(W) \mbox{ in } H^s(W),\\
&H^s_{\overline{W}}= \{y \in H^{s}(\R^n): \supp(y) \subset \overline{W}\}.
\end{align*}
We have 
\begin{align*}
(H^{s}(W))^{\ast} = \widetilde{H}^{-s}(W), \  (\widetilde{H}^{-s}(W))^{\ast} = (H^{s}(W)).
\end{align*}
If $W$ is Lipschitz and $s\geq 0$, then also
$\widetilde{H}^{s}(W) = H^{s}_{\overline{W}}$.

\subsection{Well-posedness and the Dirichlet-to-Neumann map}

We recall the well-posedness results and the weak definition of the Dirichlet-to-Neumann map for fractional Schrödinger operators without zero eigenvalues. This follows the lines of \cite{GSU16}.

\begin{lem}[Lemma 2.3 in \cite{GSU16}]
\label{lem:well_posed}
Let $s\in (0,1)$ and $n\in \N$. Let $\Omega \subset \R^n$ be open and let $q\in L^{\infty}(\Omega)$. Consider for $u,v \in H^s(\R^n)$ the bilinear form
\begin{align*}
B_q(u,v):=((-\D)^{\frac{s}{2}}u, (-\D)^{\frac{s}{2}}v)_{\R^n} + (q u, v)_{\Omega}.
\end{align*}
Then, there exists a countable set $\Sigma \subset \R$ such that for $\lambda \in \R \setminus \Sigma$, any $F\in (\widetilde{H}^s(\Omega))^{\ast}$ and any $f\in H^{s}(\R^n)$ there exists a unique $u\in H^{s}(\R^n)$ such that
\begin{align*}
B_{q}(u,w) - \lambda(u,w) = F(w) 
\end{align*}
for $w \in \widetilde{H}^s(\Omega)$ and $u-f \in \widetilde{H}^s(\Omega)$. Moreover,
\begin{align*}
\|u\|_{H^s(\R^n)} \leq C (\|F\|_{(\tilde{H}^s(\Omega))^{\ast}} + \|f\|_{H^s(\R^n)}).
\end{align*}
\end{lem}

The solution from Lemma \ref{lem:well_posed} also agrees with the distributional solution to \eqref{eq:frac_Cal}.

With the bilinear form from Lemma \ref{lem:well_posed} at hand, it is possible to define the weak form of the generalized \emph{Dirichlet-to-Neumann map}
\begin{align*}
\Lambda_{q}: H^s(\R^n) / \widetilde{H}^s(\Omega) \rightarrow (H^s(\R^n) / \widetilde{H}^s(\Omega))^{\ast}
\end{align*}
by
\begin{align}
\label{eq:DtNa}
(\Lambda_q [f], [g]) = B_q(u_f, g).
\end{align}
Here $f,g \in H^{s}(\R^n)$ and $u_f$ is a solution to \eqref{eq:frac_Cal} such that $u-f \in \widetilde{H}^s(\Omega)$. We note that for a bounded Lipschitz domain $\Omega$, we can identify $H^s(\R^n) / \widetilde{H}^s(\Omega)$ with $H^s(\Omega_e)$. If $\Omega$ is bounded, open and Lipschitz, we will mostly drop the quotient space notation and simply write $\Lambda_q f$ instead of $\Lambda_q [f]$ (c.f. Lemma 2.4 in \cite{GSU16} for more details).

A further ingredient which will be used in the proof of Theorem \ref{thm:Lip_frac_Runge} is a generalized Alessandrini type identity for the fractional Calder\'on problem (c.f. Lemma 2.5 in \cite{GSU16}), which we briefly recall.

\begin{lem}[\cite{GSU16}, Lemma 2.5]
\label{lem:Aless}
Let $s\in (0,1)$, $n\geq 1$. Assume that $\Omega \subset \R^n$ is an open, bounded domain and that $W \subset \Omega_e$ is open. Assume that \eqref{eq:not_zero} holds. Then for two solutions $u_1, u_2$ of \eqref{eq:frac_Cal} with potentials $q_1, q_2 \in L^{\infty}(\Omega)$ and with exterior data $f_1, f_2$ we have
\begin{align*}
((q_1-q_2)u_1,u_2)_{\Omega} = ((\Lambda_{q_1}-\Lambda_{q_2})f_1, f_2)_{W}.
\end{align*}
\end{lem}

Finally, we recall the weak unique continuation property for the fractional Laplacian, which due to the antilocality of the fractional Laplacian is a global result:

\begin{lem}[\cite{GSU16}, Theorem 1.2]
\label{lem:WUCP}
Let $s\in(0,1)$ and $u\in H^{s}(\R^n)$ for $n\geq 1$. Assume that for $\Omega \subset \R^n$ open
\begin{align*}
u= 0 = (-\D)^s u \mbox{ in } \Omega.
\end{align*}
Then $u \equiv 0$.
\end{lem}

This result was proved in \cite{GSU16} based on Carleman estimates from \cite{Rue15}, c.f. also \cite{FF14} and \cite{Yu16} for unique continuation results for the fractional Laplacian.

\section{Absence of zero Dirichlet eigenvalue}\label{sec:non_zero}

In this section, we discuss our Lipschitz stability result in the setting in which zero is not a Dirichlet eigenvalue. It is obtained by sampling the Dirichlet-to-Neumann map \emph{finitely} many times. 

\begin{thm}
\label{thm:Lip_frac_Runge}
Let $s\in(0,1)$ and $n,m\in \N$. Let $\Omega \subset \R^n$ be an open, bounded and smooth domain. Suppose that $W_1, W_2 \subset \Omega_e$ are open Lipschitz domains with $\overline{W}_k \cap \overline{\Omega} = \emptyset$ for $k\in\{1,2\}$.
Assume that for some orthonormal (w.r.t. the $L^2(\Omega)$ scalar product) functions $g_1,\dots,g_m \in L^{\infty}(\Omega)$ the potentials
$q_1,q_2 \in L^{\infty}(\Omega)$ are such that
\begin{align*}
q_1,q_2 \in \spa\{g_1,\dots,g_m\} \mbox{ and } \|q_1\|_{L^{\infty}(\Omega)}, \|q_2\|_{L^{\infty}(\Omega)}\leq C_0<\infty.
\end{align*}
Furthermore, suppose that \eqref{eq:not_zero} holds for $q_1, q_2$.

Then there exist 
\begin{itemize}
\item a constant $C_1>0$ which depends only on the geometry of the domains $W_k$, $k\in\{1,2\}$, on the functions $g_1,\dots, g_m$, on the constant $C_0$ and on $m$,
\item and
data $\{\tilde{f}_{1}^{(k)},\dots \tilde{f}_{m}^{(k)}\} \subset \widetilde{H}^{s}(W_k)$ with $\| \tilde{f}_j^{(k)} \|_{\widetilde{H}^{s}(\Omega)} = 1$,
\end{itemize} 
such that
\begin{align*}
\|q_1-q_2\|_{L^{\infty}(\Omega)} \leq C_1 \inf \left\{ \|(\Lambda_{q_1} - \Lambda_{q_2})\tilde{f}^{(1)}\|_{H^{-s}(W_{2})},  \|(\Lambda_{q_1} - \Lambda_{q_2})\tilde{f}^{(2)}\|_{H^{-s}(W_{1})} \right\}.
\end{align*} 
Here $\tilde{f}^{(k)}:=(\tilde{f}^{(k)}_1, \dots, \tilde{f}^{(k)}_m)$ and $(\Lambda_{q_1}-\Lambda_{q_2}) \tilde{f}^{(k)} := ((\Lambda_{q_1}-\Lambda_{q_2})\tilde{f}^{(k)}_1, \dots, (\Lambda_{q_1}-\Lambda_{q_2}) \tilde{f}^{(k)}_m) $ for $k\in \{1,2\}$.
\end{thm}

\begin{rmk}
\label{rmk:ortho_linind} 
With no loss of generality and by arguing with a Gram-Schmidt orthonormalization process we can relax the orthonormality requirement on the basis $g_1, \dots, g_m$ into a linearly independence one.
\end{rmk}

\begin{rmk}
\label{rmk:dependence}
We emphasize that the dependences of the constant $C_1>0$ can be made explicit in terms of the quantities $W_k$, $k\in\{1,2\}$, on $C_0>0$ and in terms of the functions $g_1,\dots, g_m$ (c.f. Section \ref{sec:non_zero}). It depends exponentially on the ``complexity" of the setting, c.f. Section \ref{sec:opti}. 
\end{rmk}

\begin{rmk}
\label{rmk:finite}
We remark that the functions $f_1^{(k)}, \dots, f_m^{(k)}$ with $k \in \{1,2\}$ are obtained as ``control functions" for the Runge approximation of the basis to $g_1, \dots, g_m$ by solutions to a fractional Schrödinger equation with potential $q_k$. Thus, in our argument, the functions $f_1^{(k)}, \dots, f_m^{(k)}$ with $k \in \{1,2\}$ \emph{depend} on the choice of the potentials $q_k$. 

In particular, we emphasize that although the stability estimate is formulated in terms of a \emph{finite} sample of the full Dirichlet-to-Neumann map, this should be mainly viewed as of theoretical interest (rather than of practical relevance), since in general the functions $\tilde{f}^{(k)}$ with $k\in \{1,2\}$ are \emph{unknown}. Hence, although one could easily turn the stability argument leading to Theorem \ref{thm:Lip_frac_Runge} into a recovery algorithm, we do not formulate it explicitly here.

The Lipschitz stability property itself however \emph{should} be regarded as of practical relevance, since it improves the very slow logarithmic stability from the general set-up to much better Lipschitz bounds in the finite dimensional setting (at the price of large constants, which depend on the complexity of the problem).
\end{rmk}

For the sake of completeness, we formulate below the Lipschitz stability estimate from full boundary data, i.e. from the knowledge of the (generalized) Dirichlet-to-Neumann map. The proof is a straightforward consequence of Theorem \ref{thm:Lip_frac_Runge}.

\begin{cor}
\label{thm:Lip_frac_Runge_fulldata}
Under the hypotheses of Theorem \ref{thm:Lip_frac_Runge} we also have that 
\begin{align*}
\|q_1-q_2\|_{L^{\infty}(\Omega)} \leq C_1 \|\Lambda_{q_1} - \Lambda_{q_2}\|_{*}.
\end{align*} 
Here $\|\Lambda_{q_1} - \Lambda_{q_2}\|_{*}:=
\sup\left\{\int\limits_{\R^n} ((\Lambda_{q_1}- \Lambda_{q_2}) f_1) f_2 dx: \ \|f_j\|_{\widetilde{H}^{s}(W_k)}=1\right\}$.
\end{cor}

We note that the proof of Theorem \ref{thm:Lip_frac_Runge} is more direct than in the case of the classical Calder\'on problem. Instead of constructing singular solutions, it here suffices to argue by Runge approximation only.

In order to formulate the auxiliary results on Runge approximation in Section \ref{sec:quant_runge}, we first recall that in the setting in which zero is not a Dirichlet eigenvalue of the equation \eqref{eq:frac_Cal}, i.e., in the case that \eqref{eq:not_zero} holds, we can in particular consider the Poisson operator
\begin{align}
\label{eq:Poisson}
P_q: H^{s}_{\overline{W}} \rightarrow H^s(\R^n), \
f \mapsto u.
\end{align}
Here $u$ denotes the solution of \eqref{eq:frac_Cal} with exterior data $f$. 
 
In Section \ref{sec:finite} we then present the proof of Theorem \ref{thm:Lip_frac_Runge} and discuss several examples.

\subsection{Quantitative Runge approximation}
\label{sec:quant_runge}

The article \cite{RS17} established stability for the fractional Calder\'on problem by quantifying the uniqueness results of \cite{GSU16}. Here a crucial ingredient was a quantitative Runge approximation property which we will also rely on in the sequel:

\begin{thm}[\cite{RS17}, Theorem 1.4]
\label{thm:Runge_frac}
Let $s\in(0,1)$ and $n\in \N$.
Let $\Omega \subset \R^n$ be an open, bounded, smooth domain. Let $W \subset \Omega_e$ be an open Lipschitz domain with $\overline{\Omega}\cap \overline{W}= \emptyset$. Assume that $q\in L^{\infty}(\Omega)$ is such that $0$ is not a Dirichlet eigenvalue of $(-\D)^s + q$ in $\Omega$. 
Let $P_q$ denote the Poisson operator from \eqref{eq:Poisson} and let $r_{\Omega}$ denote the restriction operator to $\Omega$.

Then, there exist constants $C>0$, $\mu>0$ such that for any $\overline{v} \in H^{s}_{\overline{\Omega}}$ there exists $f\in H^s_{\overline{W}}$ with
\begin{align*}
\|\overline{v}- r_{\Omega} P_{s}f \|_{L^2(\Omega)} \leq \epsilon \|\overline{v}\|_{H^{s}_{\overline{\Omega}}}, \  \|f\|_{H^{s}_{\overline{W}}} \leq Ce^{C \epsilon^{-\mu}} \|\overline{v}\|_{L^2(\overline{\Omega})}.
\end{align*}
The constants $C>0$ and $\mu>0$ only depend on the geometries of $\Omega, W$ and on $n,s$, $\|q\|_{L^{\infty}(\Omega)}$.
\end{thm}

We recall that this result was obtained by duality to a quantitative unique continuation result for fractional Schrödinger equations. In the sequel, it will allow us to quantitatively approximate arbitrary functions by solutions to the fractional Schrödinger equation.

\subsection{Proof of Theorem \ref{thm:Lip_frac_Runge}}
\label{sec:finite}

With the Runge approximation result of Theorem \ref{thm:Runge_frac} at hand, in this section we address the proof of Theorem \ref{thm:Lip_frac_Runge}.

\begin{proof}[Proof of Theorem \ref{thm:Lip_frac_Runge}] 
We choose $2 m$ functions $h_{1}^{(k)},\dots,h_{m}^{(k)} \in \widetilde{H}^s(\Omega)$, $k\in\{1,2\}$, such that the matrix
\begin{align*}
M:=
\begin{pmatrix} ( g_1, h_{1}^{(1)} h_{1}^{(2)}  )_{\Omega}& \dots & ( g_m, h_{1}^{(1)} h_{1}^{(2)} )_{\Omega}\\
\vdots & \vdots & \vdots\\
 ( g_1, h_{m}^{(1)}h_{m}^{(2)} )_{\Omega} & \dots & ( g_m, h_{m}^{(1)} h_{m}^{(2)} )_{\Omega}
\end{pmatrix}
\end{align*}
is invertible. This can always be ensured by for example defining $h_{1}^{(1)}h_{1}^{(2)},\dots, h_{m}^{(1)} h_{m}^{(2)}$ to be a (regularized) approximation of the  basis $g_{1},\dots,g_m$ itself.
Then, by the Runge approximation result of Theorem \ref{thm:Runge_frac}, there exist functions $u_j^{(1)}, u_j^{(2)}$ with $j\in \{1,\dots,m\}$ such that 
the following quantitative approximation property holds: For $k \in \{1,2\}$ there exist boundary data $f_{j}^{(k)} \in H^s_{\overline{W}_k}$ and solutions $u_{j}^{(k)}= h_j^{(k)} + r_{j}^{(k)}$ of the fractional Schrödinger equations 
\begin{align*}
((-\D)^s + q_k)u_{j}^{(k)} & = 0 \mbox{ in } \Omega, \ u_{j}^{(k)} = f_{j}^{(k)} \mbox{ in } \Omega_e
\end{align*}
with
\begin{align}
\label{eq:approx}
\|u_{j}^{(k)} - h_j^{(k)}\|_{L^2(\Omega)}\leq \epsilon \|h_j^{(k)}\|_{H^s_{\overline{\Omega}}}, \ \|f_{j}^{(k)}\|_{H^{s}_{\overline{W}_k}} \leq C e^{C \epsilon^{-\mu}} \|h_{j}^{(k)}\|_{L^2(\Omega)}, \ j \in \{1,\dots, m\}, 
\end{align}
for some constants $C>1$ and $\mu>0$ (which have the dependences which were explained in Theorem \ref{thm:Runge_frac}). \\
Inserting the functions $u_j^{(k)}$ into the generalized  Alessandrini's identity from Lemma \ref{lem:Aless} then yields
\begin{align}
\label{eq:Al}
\begin{split}
((q_1 - q_2) h_j^{(1)}, h_j^{(2)})_{\Omega} 
&= -((q_1-q_2) r_j^{(1)}, h_j^{(2)})_{\Omega} 
-((q_1-q_2) r_j^{(2)}, h_j^{(1)})_{\Omega} 
-((q_1-q_2) r_j^{(2)}, r_j^{(1)})_{\Omega} \\
& \quad + ((\Lambda_{q_1}- \Lambda_{q_2})f_j^{(1)}, f_j^{(2)})_{\Omega_e}.
\end{split}
\end{align}
Thus,
\begin{align*}
\left| ((q_1 - q_2) h_j^{(1)}, h_j^{(2)})_{\Omega}  \right|
&\leq \left| ((q_1-q_2) r_j^{(1)}, h_j^{(2)})_{\Omega}  \right|
+ \left| ((q_1-q_2) r_j^{(2)}, h_j^{(2)})_{\Omega} \right|
+ \left| ((q_1-q_2) r_j^{(1)}, r_j^{(2)})_{\Omega} \right| \\
& \quad +  \| (\Lambda_{q_1}-\Lambda_{q_2}) f_j^{(1)}\|_{H^{-s}(W_2)} \|f_j^{(2)}\|_{H^{s}_{\overline{W}_2}}.
\end{align*}
Using this together with the assumption that
\begin{align*}
q_1-q_2 = \sum\limits_{j=1}^{m} a_j g_j,
\end{align*}
and the estimates from \eqref{eq:approx}, leads to $m$ inequalities of the form
\begin{align}
\label{eq:est_f}
\begin{split}
\left| \sum\limits_{j=1}^{m} a_j ( g_j ,h_{l}^{(1)}h_{l}^{(2)} )_{\Omega} \right|
&\leq C \|(\Lambda_{q_1}-\Lambda_{q_2})\tilde{f}_l^{(1)}\|_{H^{-s}(W_2)} e^{C \epsilon^{-\mu}} \|h_{l}^{(1)}\|_{L^2(\Omega)} \|h_{l}^{(2)}\|_{L^2(\Omega)} \\
& \quad + C \epsilon \|q_1-q_2\|_{L^{\infty}(\Omega)} \|h_{l}^{(1)}\|_{H^{s}_{\overline{\Omega}}}\|h_{l}^{(2)}\|_{H^{s}_{\overline{\Omega}}},\quad l \in\{1,\dots,m\},
\end{split}
\end{align}
where $\tilde{f}_l^{(1)}(x):= \frac{f_l^{(1)}(x)}{\|f_l^{(1)}\|_{\tilde{H}^s(W_1)}}$.
Abbreviating $h_{l}:= h_{l}^{(1)} h_{l}^{(2)}$ for $l\in \{1,\dots,m\}$, we observe that for each $l\in\{1,\dots,m\}$ the estimate \eqref{eq:est_f} can be rewritten as the following vector valued system of equations
\begin{align*}
\begin{pmatrix} ( g_1, h_{1} )_{\Omega}& \dots & ( g_m, h_{1} )_{\Omega}\\
\vdots & \vdots & \vdots\\
(g_1, h_{m} )_{\Omega} & \cdots & ( g_m, h_{m} )_{\Omega}
\end{pmatrix}
\begin{pmatrix}
a_1 \\ \vdots\\ a_m
\end{pmatrix} = v_1(\epsilon) + v_{2}(\epsilon),
\end{align*}
where $v_1(\epsilon), v_2(\epsilon) \in \R^n$ obey the bounds
\begin{align*}
|v_1(\epsilon)| &\leq C e^{C \epsilon^{-\mu}}  \sup\limits_{l \in \{1,\dots,m\}} \|(\Lambda_{q_1}-\Lambda_{q_2}) \tilde{f}_l^{(1)}\|_{H^{-s}(W_2)} \|h_{l}^{(1)}\|_{L^2(\Omega)} \|h_{l}^{(2)}\|_{L^2(\Omega)}\\ 
|v_2(\epsilon)| &\leq C \epsilon \|q_1-q_2\|_{L^{\infty}(\Omega)} \sup\limits_{l \in \{1,\dots,m\}}\|h_{l}^{(1)}\|_{H^{s}_{\overline{\Omega}}}\|h_{l}^{(2)}\|_{H^{s}_{\overline{\Omega}}}.
\end{align*}
By virtue of the assumed invertibility of the matrix $M$, we infer that for $a=(a_1,\dots, a_m)$
\begin{align}
\label{eq:gen}
\begin{split}
\|a\| 
&\leq C_M e^{C \epsilon^{-\mu}} \sup\limits_{l \in \{1,\dots,m\}} \|(\Lambda_{q_1}-\Lambda_{q_2}) \tilde{f}_l^{(1)}\|_{H^{-s}(W_2)}  \|h_{l}^{(1)}\|_{L^2(\Omega)} \|h_{l}^{(2)}\|_{L^2(\Omega)} \\
& \quad + C_M \epsilon \sup\limits_{l\in \{1,\dots,m\}} \|h_{l}^{(1)}\|_{H^{s}_{\overline{\Omega}}}\|h_{l}^{(2)}\|_{H^{s}_{\overline{\Omega}}}\|a\|.
\end{split}
\end{align}
Here $\|a\|:= \|a\|_{\ell^{\infty}}$; the subscript $M$ in the constant $C_M>0$ emphasizes that this constant in particular depends on the invertibility properties of the matrix $M$.
We note the last term on the right hand side of \eqref{eq:gen} can be absorbed into the left hand side of \eqref{eq:gen} provided that $C_M \epsilon \sup\limits_{l\in\{1,\dots,m\}}  \|h_{l}^{(1)}\|_{H^{s}_{\overline{\Omega}}}\|h_{l}^{(2)}\|_{H^{s}_{\overline{\Omega}}}  < 1$. 
As a consequence, setting 
\begin{align*}
L_{0}:=\sup\limits_{l\in\{1,\dots,m\}}  \|h_{l}^{(1)}\|_{H^{s}_{\overline{\Omega}}}\|h_{l}^{(2)}\|_{H^{s}_{\overline{\Omega}}}, \
L_{1}:=\sup\limits_{l\in\{1,\dots,m\}} \|h_{l}^{(1)}\|_{L^2(\Omega)} \|h_{l}^{(2)}\|_{L^2(\Omega)} ,
\end{align*}
 and choosing $\epsilon = \frac{1}{2}(C_M L_0)^{-1} $, we have 
\begin{align*}
\|a\| \leq C_M L_{1} e^{C C_M^{\mu} L_{0}^{\mu}}\|(\Lambda_{q_1}-\Lambda_{q_2}) \tilde{f}^{(1)}\|_{H^{-s}(W_2)} .
\end{align*}
Noting that $\Lambda_q$ is a symmetric operator, we observe that in \eqref{eq:Al}, we could also have reversed the roles of $f_j^{(1)}, f_j^{(2)}$, which would have led to
\begin{align*}
\|a\| \leq C_M L_{1} e^{C C_M^{\mu} L_{0}^{\mu}}\|(\Lambda_{q_1}-\Lambda_{q_2}) \tilde{f}^{(2)}\|_{H^{-s}(W_1)}. 
\end{align*}  
This concludes the argument for Theorem \ref{thm:Lip_frac_Runge}.
\end{proof}

\begin{rmk}
\label{rmk:dependences_1} 
The above proof keeps quite explicitly track of the dependences of the constants on the various parameters. It is an interesting question, to optimize the choice of the basis functions $g_j$ and the dual test functions $h_{l}^{(k)}$ for $k\in\{1,2\}$.
\end{rmk}

\begin{example}
\label{ex:pc}
As an example of Theorem \ref{thm:Lip_frac_Runge}, one may consider the case of piecewise constant potentials, which has been studied in \cite{BdHQ13} for the classical Schrödinger case. In this case we have a (up to null sets) disjoint covering $\Omega = \bigcup\limits_{j=1}^{N} D_j$ by bounded Lipschitz sets $D_j \subset \Omega$ and $g_j:= \chi_{D_j}$, where $\chi_{D_j}$ are the characteristic functions of the sets $D_j$, i.e., 
\begin{align*}
\chi_{D_j}(x):=\left\{ \begin{array}{ll} 1 \mbox{ if } x \in D_j, \\ 0 \mbox{ else}. \end{array} \right.
\end{align*}
If $s< \frac{1}{2}$ we can then choose the functions $h_j^{(k)}$ from the proof of Theorem \ref{thm:Lip_frac_Runge} to be the (normalized) characteristic functions of the domains $D_j$, while for $s\geq \frac{1}{2}$, we can for instance choose smoothed out versions of these. Choosing for $\Omega = [0,1]^n$ the sets $D_j $ to be translates of the interval $[0,N^{-1/n}]^n$ then shows that the constant $L_0^{\mu}$ is proportional to $N^{\tilde{\mu}}$ for some constant $\tilde{\mu}>0$. This is consistent with the results of \cite{Ron06} (c.f. also Section \ref{sec:opti}).
\end{example}

\begin{example}
\label{ex:pl}
From an applications point of view another interesting example of a system of basis functions $\{g_1,\dots, g_m\}$ corresponds to piecewise affine basis functions, as discussed in \cite{AdHGS17} for the classical Schrödinger case. Considering a partition of $\Omega$ as in Example \ref{ex:pc}, we here have $m=Nn$ and
\begin{align*}
\{g_1, \dots, g_m\}= \{\chi_{D_1}, \dots, \chi_{D_N}, x_1 \chi_{D_1}, \dots, x_1 \chi_{D_N}, x_2 \chi_{D_2}, \dots, x_2 \chi_{D_2}, \dots, x_n \chi_{D_N}, \dots, x_n \chi_{D_N}\}.
\end{align*}
Another canonical set of basis functions $\{g_1,\dots, g_m\}$ for instance consists of a subset of the trigonometric functions.
\end{example}

\begin{rmk}
\label{rmk:rough}
We emphasize that the choice of the space $q\in L^{\infty}(\Omega)$ does not play a major role in our arguments. Instead of using the $L^{\infty}$ based Runge approximation results from above, it would also have been possible to work with the space $Z^{-s}_0(\R^n)$ and the approximation results from Lemmas 8.1 and 8.2 in \cite{RS17}. This would for instance have allowed us to treat potentials $q$ in finite dimensional subspaces of $L^{\frac{n}{2s}}(\Omega)$. Since there are no major changes in our arguments (except for having to rely on the stronger results of Lemmas 8.1 and 8.2 from \cite{RS17} instead of Theorem \ref{thm:Runge_frac}) and since such a more general set-up would introduce additional technicalities, we have opted against formulating the corresponding results.
\end{rmk}

\section{Presence of Zero Eigenvalue}
\label{sec:zero}

In the case in which zero is a Dirichlet eigenvalue of the problem \eqref{eq:frac_Cal}, i.e. in the case in which \eqref{eq:not_zero} is violated, we have to modify the functional analytic set-up of the argument from Section \ref{sec:non_zero} slightly. In particular, we
can no longer work with the Dirichlet-to-Neumann map, but have to consider the Cauchy data associated with the problem. 

Let us make this more precise: For $\Omega \subset \R^n$ open, bounded and smooth and $W \subset \Omega_e$ open and Lipschitz, we consider the inhomogeneous Dirichlet problem
\begin{align}
\label{eq:inhom_Dir}
\begin{split}
((-\D)^s + q) u & = F \mbox{ in } \Omega,\\
u & = f \mbox{ in } \Omega_e,
\end{split}
\end{align}
with $F \in (\widetilde{H}^{s}(\Omega))^{\ast}=H^{-s}(\Omega)$, $f\in \widetilde{H}^s(W)$.
By virtue of the results of Grubb \cite{Gr15} (which are formulated for smooth domains, hence our restriction to this class of domains), the inhomogeneous, zero Dirichlet data problem
\begin{align}
\label{eq:inhom_Dir_a}
\begin{split}
((-\D)^s + q) \tilde{u} & = \tilde{F} \mbox{ in } \Omega,\\
\tilde{u} & = 0 \mbox{ in } \Omega_e,
\end{split}
\end{align}
with $\tilde{F} \in H^{-s}(\Omega)$ gives rise to a Fredholm operator
\begin{align*}
\tilde{T}: \widetilde{H}^s(\Omega) \ni \tilde{u} \mapsto \tilde{F} \in H^{-s}(\Omega).
\end{align*}
Due to the observations in Lemma 2.3 in \cite{GSU16}, this operator is of index zero (indeed, if $q=0$, it is of index zero and multiplication with $q \in L^{\infty}(\Omega)$ is a compact perturbation, which does not change the index of $\tilde{T}$). We note that the problem \eqref{eq:inhom_Dir} can be reduced to \eqref{eq:inhom_Dir_a} by extending $f$ by zero and defining $\tilde{u}:= u- e_0 f$, which leads to $\tilde{F}=F-(-\D)^s (e_0 f)$. Here $e_0$ denotes the extension operator by zero. We have used that $\widetilde{H}^s(W)=H^s_{\overline{W}}$ for the Lipschitz domain $W \subset \Omega_e$.

The reduction of \eqref{eq:inhom_Dir} to \eqref{eq:inhom_Dir_a} and the results from \cite{Gr15} allow us to invoke the Fredholm alternative (c.f. for instance Theorem 2.27 in \cite{McLean}). In particular, this entails that there exists a finite dimensional space
\begin{align}
\label{eq:eigen}
Z_2:=\{z \in \widetilde{H}^s(\Omega): \ ((-\D)^s + q) z = 0 \mbox{ in } \Omega\} \subset \widetilde{H}^s(\Omega) \subset L^2(\Omega),
\end{align}
such that the problem \eqref{eq:inhom_Dir} is solvable, if and only if
\begin{align}
\label{eq:solvability}
(F,z)_{\Omega} + (f,(-\D)^s z)_{W} = 0 \mbox{ for all } z \in Z_2.
\end{align}
Here we used that $\supp(f) \subset \overline{W}$ and view the identity \eqref{eq:solvability} as a duality pairing (in the usual sense, c.f. for instance Theorem 3.14 in \cite{McLean}). We remark that since $\widetilde{H}^{s}(\Omega)= H^s_{\overline{\Omega}}$, we can also interpret $Z_2$ as a subset of $H^{s}(\R^n)$ (and not only of $H^{s}(\Omega)$).
Seeking to define the operator $A:= \iota r_{\Omega}( P_{q}-Id)$ in analogy to the operator from \cite{RS17}, where 
\begin{itemize}
\item $r_{\Omega}$ denotes the restriction operator onto $\Omega$,
\item $\iota$ is the (compact) inclusion operator of $H^{s}(\R^n)$ into $L^2(\Omega)$, 
\item $P_q$ is a Poisson type operator for \eqref{eq:inhom_Dir} (which is defined in \eqref{eq:Pq} below),
\item and $Id$ denotes the identity operator, 
\end{itemize}
the following functional analytic set-up arises naturally: Motivated by the solvability properties of \eqref{eq:inhom_Dir} for a fixed open Lipschitz set $W \subset \Omega_e$, we introduce the following orthogonal decompositions
\begin{align*}
\widetilde{H}^{s}(W)= Z_1 \perp H_1, \ L^2(\Omega) = Z_2 \perp H_2.
\end{align*}
Here $Z_2$ is the function space from \eqref{eq:eigen}, while $Z_1$ is defined as the orthogonal complement of
\begin{align}
\label{eq:eigen_1}
H_{1}= \{y \in \widetilde{H}^{s}(W): \ (y,(-\D)^s z)_{W}=0, \ z \in Z_2 \} \subset \widetilde{H}^s(W).
\end{align}
By the Fredholm alternative, both $Z_1$ and $Z_2$ are finite dimensional.

In this setting the Poisson operator to the problem \eqref{eq:inhom_Dir} with $F=0$ is now defined as the mapping
\begin{align}
\label{eq:Pq}
P_q: H_1 \rightarrow H^s(\R^n)/Z_2, \ f \mapsto [u].
\end{align}
Here $u$ solves \eqref{eq:inhom_Dir} with $F=0$ and boundary data $f\in H_1$ and $H^s(\R^n)/Z_2$ denotes the quotient space of $H^{s}(\R^n)$ and $Z_2$ (for whose definition we use that $Z_2 \subset H^s(\R^n)$). The notation $[u]$ denotes the equivalence class in $H^s(\R^n)/Z_2$ containing the element $u$. By an argument as in the proof of Lemma 2.3 in \cite{GSU16} (now formulated in the quotient space $H^s(\R^n)/Z_2$) the operator $P_q$ is bounded and linear.

Having established this notation, by Sobolev embedding, we define the compact, linear operator
\begin{align*}
A: H_1  \subset \widetilde{H}^{s}(W) \rightarrow H_2 \subset L^2(\Omega), \ 
f \mapsto \iota r_{\Omega} P_q f,
\end{align*}
which by the previous considerations is a well-defined operator. With slight abuse of notation we have here identified $H_2$ with $L^2(\Omega)/Z_2$ (in particular, with this identification, we do not use the notation of equivalence classes any longer but simply write $u$ instead of $[u]$). We remark that by the disjointness of the sets $\Omega, W$ the operator $A$ can also be written as $A f = \iota r_{\Omega} (P_qf -f)$. Its Banach space adjoint $A^t$ can be computed to be $A^t v = (-\D)^s w|_{W}$, where $v \in H_2$ and $w \in H_2$ are related through the dual equation
\begin{align}
\label{eq:dual}
\begin{split}
((-\D)^s + q) w & = v \mbox{ in } \Omega,\\
w & = 0 \mbox{ in } \Omega_e.
\end{split}
\end{align}
We remark that the dual problem is solvable, since $v \in H_2$. In order to have well-posedness, as in the definition of $P_q$, we have also chosen to define the solution map to $\eqref{eq:dual}$ to involve the orthogonal projection onto $H_2$, whence $w\in H_2$. We will always use this implicitly in the sequel.

Hence, by general functional analysis the Hilbert space adjoint of $A$ is given by
\begin{align*}
A^{\ast}: H_2 \rightarrow H_1 ,\
v \mapsto R_{s}^{-1}(-\D)^s w|_{W},
\end{align*}
where $R_s$ denotes the Riesz isomorphism from $\widetilde{H}^{s}(W)$ to $H^{-s}(W)$. 

Due to the presence of zero as an eigenvalue to \eqref{eq:frac_Cal}, we can no longer define the Dirichlet-to-Neumann operator as in \eqref{eq:DtN}. Instead we consider the \emph{Cauchy data} set associated with the problem \eqref{eq:frac_Cal}. For the potentials $q_i \in L^{\infty}(\Omega)$ with $i\in\{1,2\}$ and a set $W \subset \Omega_e$ this amounts to considering the sets
\begin{align*}
\mathcal{C}_i
& := \{(f,g) \in \widetilde{H}^{s}(W)\times H^{-s}(W): f=u_i|_{W}, \ g= (-\D)^s u_i|_{W} \\
& \quad \quad \mbox{ for some solution $u_i \in H^{s}(\R^n)$ to \eqref{eq:frac_Cal} with potential $q_i$} \}.
\end{align*}
Equipped with the $\widetilde{H}^{s}(W) \times H^{-s}(W)$ norm, which for convenience we abbreviate by a subindex $H$, the Cauchy data form a closed subspace of the Hilbert space $\widetilde{H}^{s}(W) \times H^{-s}(W)$. We define the distance of the two Cauchy data sets to be
\begin{align*}
d(\mathcal{C}_1, \mathcal{C}_2)
= \max\left\{ \sup\limits_{h\in \mathcal{C}_2, h \neq 0} \inf\limits_{k \in \mathcal{C}_1} \frac{\|h-k\|_{H}}{\|h\|_{H}},  \sup\limits_{h\in \mathcal{C}_1, h \neq 0} \inf\limits_{k \in \mathcal{C}_2} \frac{\|h-k\|_{H}}{\|h\|_{H}} \right\} .
\end{align*} 
We recall that if zero is not an eigenvalue of the fractional Schrödinger operator, this distance is comparable to the norm of the difference of the Dirichlet-to-Neumann maps $\Lambda_{q_1}$, $\Lambda_{q_2}$.

In presenting the analogue of the results of Section \ref{sec:non_zero} in the outlined more general set-up in which zero is allowed to be an eigenvalue of \eqref{eq:frac_Cal}, we begin by  formulating the full data analogue of Theorem \ref{thm:Lip_frac_Runge}.

\begin{thm}
\label{thm:Lip_frac_Runge_eigen}
Let $n,m \in \N$ and $s\in(0,1)$.
Let $\Omega \subset \R^n$ be an open, bounded, smooth domain. Suppose that $W_1, W_2 \subset \Omega_e$ are Lipschitz with $\overline{W_j} \cap \overline{\Omega} = \emptyset$ for $j\in \{1,2\}$.
Let $q_1,q_2 \in L^{\infty}(\Omega)$ with
\begin{align*}
q_1, q_2 \in \spa\{g_1,\dots,g_m\} \mbox{ and } \|q_1\|_{L^{\infty}(\Omega)}, \|q_{2}\|_{L^{\infty}(\Omega)} \leq C_0 <\infty,
\end{align*}
for some orthonormal (w.r.t. the $L^2(\Omega)$ scalar product) functions $g_1,\dots, g_m \in L^{\infty}(\Omega)$.
Then there exists a constant $C_2>0$, which depends only on the geometry of the domains $W_1, W_2$, on the functions $g_1,\dots,g_m$ and on the constant $C_0>0$, 
such that
\begin{align*}
\|q_1-q_2\|_{L^{\infty}(\Omega)} \leq C_2 d(\mathcal{C}_1, \mathcal{C}_2).
\end{align*} 
\end{thm}

\begin{rmk}
\label{rmk:dependences_3} 
As in the case of Theorem \ref{thm:Lip_frac_Runge}, the dependence of $C_2$ on $N$ can be made quite explicit. It is essentially determined by the geometry of the domains $W_1, W_2$, by the constant $C_0>0$ and by the choice of the functions $g_1,\dots,g_m$.
\end{rmk}

Similarly, as in Theorem \ref{thm:Lip_frac_Runge}, it is also possible to formulate a \emph{finite} measurement result for the setting in which zero is allowed to be an eigenvalue of the fractional Schrödinger equation \eqref{eq:frac_Cal}. In order to formulate this, we introduce a suitable \emph{finite measurement Cauchy data set}, which is a finite subset of the full Cauchy data set:
For given functions $p_1,\dots, p_m \in \widetilde{H}^s(W)$ and $m\in \N$ we set $p:=(p_1,\dots, p_m) \in (\widetilde{H}^s(W))^n$ and define 
\begin{align*}
\mathcal{C}_i(p)
&:=\{(p_l,g) \in \{p_1,\dots,p_m\} \times H^{-s}(W) \subset \widetilde{H}^s(W) \times H^{-s}(W), \ (p_l,g)\in \mathcal{C}_i \}
 \subset \mathcal{C}_i.
\end{align*}
As in \cite{AdHGS17} we measure the distance between finite sets of Cauchy data by the following quantity:
\begin{align*}
d(\mathcal{C}_1(p), \mathcal{C}_2(p))
= \max\left\{ \sup\limits_{h\in \mathcal{C}_2(p), h \neq 0} \inf\limits_{k \in \mathcal{C}_1(p)} \frac{\|h-k\|_{H}}{\|h\|_{H}},  \sup\limits_{h\in \mathcal{C}_1(p), h \neq 0} \inf\limits_{k \in \mathcal{C}_2(p)} \frac{\|h-k\|_{H}}{\|h\|_{H}} \right\}.
\end{align*} 

With this notation at hand, the finite Cauchy data result can be formulated as follows:

\begin{prop}
\label{prop:finite_Cauchy}
Under the hypotheses of Theorem \ref{thm:Lip_frac_Runge_eigen} there exist
\begin{itemize}
\item a constant $C_2>0$, which depends only on the geometry of the domains $W_1, W_2$, on the functions $g_1,\dots,g_m$ and on the constant $C_0>0$,
\item functions $f_1^{(k)},\dots, f_m^{(k)} \in \tilde{H}^s(W_k)$ with $k\in \{1,2\}$,
\end{itemize}
such that
\begin{align*}
\|q_1-q_2\|_{L^{\infty}(\Omega)} \leq C_2 d(\mathcal{C}_1(f^{(1)}), \mathcal{C}_2(f^{(2)})). 
\end{align*}  
Here $f^{(k)}:=(f_1^{(k)}, \dots, f_m^{(k)})$, where $k\in \{1,2\}$.
\end{prop}

In the sequel, we derive these results by slight adaptations of the arguments from Section \ref{sec:non_zero}. In particular, this necessitates derivations of analogues of the qualitative and quantitative Runge approximation property, which are derived in Sections \ref{sec:qual_Runge_zero} and \ref{sec:quant_Runge_zero}. Relying on these auxiliary results, which might be interesting on their own right, we then present the argument for Theorem \ref{thm:Lip_frac_Runge_eigen} in Section \ref{sec:Proof_zero}.

\subsection{Runge approximation}

In order to derive the stability properties claimed above, we extend the Runge approximation property from Theorem \ref{thm:Runge_frac} to the case in which the potential $q$ is allowed to have zero as a Dirichlet eigenvalue for \eqref{eq:frac_Cal}, i.e. the condition \eqref{eq:not_zero} is violated. 

\subsubsection{Qualitative Runge approximation}
\label{sec:qual_Runge_zero}
Before turning to the quantitative Runge approximation property, we reprove its qualitative variant.
The main result here reads:

\begin{lem}
\label{lem:qual_Runge_zero}
Let $n\in \N$, $s\in (0,1)$.
Let $\Omega \subset \R^n$ be an open, bounded and smooth domain. Assume that $W \subset \Omega_e$ is open with $\overline{\Omega}\cap \overline{W}= \emptyset$. Let $q\in L^{\infty}(\Omega)$. 
Then, the set 
\begin{align*}
\mathcal{R}:=\{r_{\Omega} P_q f  + z: \ f \in H_1 , \ z \in Z_2 \}
\end{align*}
is dense in $L^2(\Omega)$.
\end{lem}

\begin{proof}
In order to show the density result, by Hahn-Banach, it suffices to prove that any function $h \in L^2(\Omega)$ with $(h,g)_{\Omega}=0$ for all $g \in \mathcal{R}$ has to vanish identically. Given such a function $h$, we may immediately assume that $(h,z)_{\Omega}=0$ for all $z\in Z_2$, as $f=0 \in H_1$. In particular this entails that the dual problem
\begin{align*}
((-\D)^s + q) w &= h \mbox{ in } \Omega, \\
w & = 0 \mbox{ in } \Omega_e,
\end{align*}
is solvable with $w\in H_2$. Hence,
\begin{align*}
0 & = (h,P_q f)_{\Omega} = (((-\D)^s + q)w, P_q f)_{\Omega}
= ((-\D)^s w|_{W}, f)_{W} \mbox{ for all } f \in H_1.
\end{align*}
As a consequence, $(-\D)^s w|_{W} \in (H_1)^a$, where $(H_1)^a$ denotes the annihilator of $H_1$. By definition of the space $H_1$ this however entails that there exist $(\lambda_j, z_j) \in \R \times Z_2$, $j\in \{1,\dots,m\}$ and $m\in \N$, such that
\begin{align*}
(-\D)^s w|_{W} = \sum\limits_{j=1}^{m} \lambda_j (-\D)^s z_j|_{W}.
\end{align*}
Since also $w=0= \sum\limits_{j=1}^{m}\lambda_j z_j$ on $W$, by the weak unique continuation property (c.f. Lemma \ref{lem:WUCP}), this entails that $w= \sum\limits_{j=1}^{m} \lambda_j z_j$, whence $w  \in Z_2 \cap H_2$. From this we however infer $w\equiv 0$ and thus $h\equiv 0$, which concludes the argument.
\end{proof}

\begin{rmk}
\label{rmk:dense}
By orthogonality of the spaces $H_2, Z_2$, Lemma \ref{lem:qual_Runge_zero} also directly implies that for $W \subset \Omega_e$ with $\overline{\Omega}\cap \overline{W}= \emptyset$ the set
\begin{align*}
\mathcal{R}:= \{r_{\Omega} P_q f: \ f\in H_1 \}
\end{align*}
is dense in $H_2$.
\end{rmk}

\subsubsection{Quantitative Runge approximation}
\label{sec:quant_Runge_zero}
Relying on the qualitative Runge approximation from above,
as the main result of this subsection, we prove the following quantitative Runge approximation property.

\begin{thm}
\label{thm:quant_Runge_frac}
Let $s\in(0,1)$ and $n\in \N$.
Let $\Omega \subset \R^n$ be an open, bounded, smooth domain. Let $W \subset \Omega_e$ be a Lipschitz domain with $\overline{\Omega}\cap \overline{W}= \emptyset$. Assume that $q\in L^{\infty}(\Omega)$. Then there exist constants $C,\mu>0$ such that for any $\overline{v} \in H^{s}_{\overline{\Omega}}$ there are $f\in H^{s}_{\overline{W}}$ and $z\in Z_2$ with
\begin{align*}
\|\overline{v}- P_{s}f - z\|_{L^2(\Omega)} \leq \epsilon \|\overline{v}\|_{H^{s}_{\overline{\Omega}}}, \  \|f\|_{H^{s}_{\overline{W}}} \leq C e^{C \epsilon^{-\mu}} \|\overline{v}\|_{L^2(\Omega)}.
\end{align*}
The constants $C,\mu$ only depend on $s,n, \Omega, W, \|q\|_{L^{\infty}(\Omega)}$.
\end{thm}

The quantitative Runge approximation property will follow from a modification of Lemma 2.3 in \cite{RS17} and the quantitative unique continuation result from \cite{RS17}. Indeed, the finite dimensional additional part, will not play any role for the cost of controlability (c.f. the arguments below, in particular Lemma \ref{lem:control}).\\

As an auxiliary step towards the proof of Theorem \ref{thm:quant_Runge_frac}, we rely on the singular value decomposition of the operators $A, A^{\ast}$. 

\begin{lem}
\label{lem:zero_eig_spec}
Let $n\in \N$, $s\in (0,1)$.
Let $\Omega \subset \R^n$ be an open, bounded and smooth domain. Let $W \subset \Omega_e$ be open with $\overline{\Omega}\cap \overline{W}= \emptyset$.
Let $A:H_1 \subset \widetilde{H}^s(W)\rightarrow H_2 \subset L^2(\Omega)$ be as above. Then $A$ is a compact, injective linear operator whose image in $H_2$ is dense. Moreover, there exists a singular value decomposition with $\{\varphi_k\}_{k \in \N} \subset H_1$, $\{\psi_k\}_{k\in \N} \subset H_2$, $\{\sigma_k\}_{k\in \N} \subset \R_+$ and $\sigma_k \geq \sigma_{k+1}\geq \dots >0$. In particular,
\begin{align*}
A \varphi_k = \sigma_k \psi_k, \ A^{\ast} \psi_k = \sigma_k \varphi_k.
\end{align*}
\end{lem}

\begin{proof}
The compactness of $A$ follows from the regularity properties of the Poisson operator $P_q$, which maps $H_1$ into a subspace of $H^s(\R^n)$, and the compactness of the embedding $\iota :\widetilde{H}^{s}(\Omega) \rightarrow L^2(\Omega)$. The injectivity is a consequence of the unique continuation principle for the fractional Laplacian. The density of the image in $H_2$ follows from the qualitative Runge approximation of Lemma \ref{lem:qual_Runge_zero} and of Remark \ref{rmk:dense}.\\
Hence, by the spectral theorem for compact, self-adjoint operators, the operator $A^{\ast} A: H_1 \rightarrow H_1$ has an orthonormal eigenvalue basis $\{\varphi_k\}_{k\in \N}$ with positive, decreasing eigenvalues $\{\mu_k\}_{k\in \N}$. For $\sigma_k:= \sqrt{\mu_k}$ we set $\psi_k:= \frac{1}{\sigma_k} A \varphi_k$ and claim that this forms a complete, orthonormal system in $H_2$. As orthonormality follows from the definition, it suffices to check completeness. This however follows from the density of the span of $\{\varphi_k\}_{k\in \N}$ in $H_1$ and the unique continuation properties of the fractional Laplacian: By Hahn-Banach, it suffices to prove that if for some $v\in H_2$ it holds that
\begin{align}
\label{eq:HB}
(\psi_k, v)_{\Omega} = 0 \mbox{ for all } k \in \N,
\end{align}
then $v\equiv 0$. But as \eqref{eq:HB} is equivalent to the condition that $(\varphi_k, A^{\ast}v)_{\widetilde{H}^s(W)}=0$ for all $k\in \N$, we infer that on the one hand $A^{\ast} v \in Z_1$. By definition of $A^{\ast}$, this implies that $(\varphi_k,(-\D)^s w|_{W})_{W}=0$ for all $k \in \N$. Here $w\in H_2$ denotes the solution of the dual problem \eqref{eq:dual} with inhomogeneity $v\in H_2$.
Thus, there exist $\alpha_i \in \R$, $z_i \in Z_2$, $i\in\{1,\dots,m\}$ such that
\begin{align*}
(-\D)^s w|_{W} = \sum\limits_{i=1}^{m} \alpha_i (-\D)^s z_i|_{W}. 
\end{align*}
As also $w= 0=\sum\limits_{i=1}^{m} \alpha_i  z_i$ on $W$, the (weak) unique continuation property of the fractional Laplacian (c.f. Lemma \ref{lem:WUCP}) implies that on the one hand $w\equiv \sum\limits_{i=1}^{m} \alpha_i  z_i \in Z_2$. However, on the other hand the assumption that $v\in H_2$ yields $w \in H_2$. Since thus $w \in Z_2 \cap H_2$, we arrive at the conclusion $w\equiv 0$ and hence $v \equiv 0$, which concludes the completeness proof. 
\end{proof}

Based on the singular value decomposition and following the lines of the argument presented in Lemma 3.3  in \cite{RS17}, we have the following auxiliary result:

\begin{lem}
\label{lem:control}
Let $n\in \N$, $s\in (0,1)$.
Let $\Omega \subset \R^n$ be an open, bounded and smooth domain. Let $W \subset \R^n$ be open with $\overline{\Omega}\cap \overline{W}= \emptyset$. Let $w \in H^{s}_{\overline{\Omega}}\cap H_2$ be a solution of the dual equation \eqref{eq:dual} with inhomogeneity $v \in H_2$ and potential $q \in L^{\infty}(\Omega)$.
Assume that for some $C>0, \mu >0$ we have
\begin{align*}
\|v\|_{H^{-s}(\Omega)} \leq C \left(\log \left( \frac{\|v\|_{L^2(\Omega)}}{\|(-\D)^s w\|_{H^{-s}(W)}}\right) \right)^{-\mu} \|v\|_{L^2(\Omega)}.
\end{align*}
Then for any $\overline{v}\in H^{s}_{\overline{\Omega}}$ and any $\epsilon>0$ there exist $f\in H_1$ and $z \in Z_2$ such that
\begin{align*}
\|\overline{v}- P_{s}f - z\|_{L^2(\Omega)} \leq \epsilon \|\overline{v}\|_{H^{s}_{\overline{\Omega}}}, \  \|f\|_{H^{s}_{\overline{W}}} \leq C \exp(C \epsilon^{-\mu}) \|\overline{v}\|_{L^2(\overline{\Omega})}.
\end{align*}
\end{lem}

\begin{proof}
The argument essentially follows as in \cite{RS17}, since the additional finite dimensional subspace $Z_2$ can be dealt with ``by hand". More precisely, given $\overline{v}\in H^{s}_{\overline{\Omega}} \subset L^2(\Omega)$, the singular value decomposition of Lemma \ref{lem:zero_eig_spec} implies that
\begin{align*}
\overline{v}= \sum\limits_{j=1}^{\infty} \beta_j \psi_j + \sum\limits_{l=1}^m \tilde{\beta}_l z_l,
\end{align*}
where $z_l \in Z_2$.
Considering the projection $\tilde{v}=\sum\limits_{j=1}^{\infty} \beta_j \psi_j$ of $\overline{v}$ onto $H_2$ first,
we define
\begin{align*}
R_{\alpha} \tilde{v} = \sum\limits_{\sigma_j \geq \alpha} \frac{\beta_j}{\sigma_j} \varphi_j \in H_1 \subset H^{s}_{\overline{W}}.
\end{align*}
Then by orthogonality (of the functions $\varphi_j$ but also of the spaces $H_2$ and $Z_2$) we obtain that
\begin{align*}
\|R_{\alpha} \tilde{v}\|_{H^s_{\overline{W}}} \leq \frac{1}{\alpha} \|\tilde{v}\|_{L^2(\Omega)} \leq \frac{1}{\alpha} \|\overline{v}\|_{L^2(\Omega)}. 
\end{align*}
Furthermore, we set $r_{\alpha}:= \sum\limits_{\sigma_j < \alpha} \beta_j \psi_j \in H_2$.
Observing that $A R_{\alpha} \tilde{v} = \sum\limits_{\sigma_j \geq \alpha} \beta_j \psi_j $, relying on Lemma \ref{lem:zero_eig_spec} and using orthogonality of $H_2$ and $Z_2$, then yields
\begin{align*}
\|A R_{\alpha}(\tilde{v})-\tilde{v}\|_{L^2(\Omega)}^2
&\leq |(\tilde{v},r_{\alpha})_{L^2(\Omega)}|
= |(\overline{v},r_{\alpha})_{L^2(\Omega)}|
\leq \|\overline{v}\|_{\widetilde{H}^{s}(\Omega)}\|r_{\alpha}\|_{H^{-s}(\Omega)}
\leq \|\overline{v}\|_{H^{s}(\R^n)}\|r_{\alpha}\|_{H^{-s}(\Omega)}\\
&\leq \|\overline{v}\|_{H^{s}_{\overline{\Omega}}}\|r_{\alpha}\|_{H^{-s}(\Omega)}
\leq \|\overline{v}\|_{H^{s}_{\overline{\Omega}}} \left(\log\left( \frac{\|r_{\alpha}\|_{L^2(\Omega)}}{\|A^{\ast} r_{\alpha}\|_{H^{s}_{\overline{W}}}} \right) \right)^{-\mu} \|r_{\alpha}\|_{L^2(\Omega)}\\
& \leq \|\overline{v}\|_{H^{s}_{\overline{\Omega}}} \left( \log \left(1/\alpha\right) \right)^{-\mu} \|r_{\alpha}\|_{L^2(\Omega)}
 \leq \|\overline{v}\|_{H^{s}_{\overline{\Omega}}} \left(\log \left(1/\alpha\right) \right)^{-\mu} \|A R_{\alpha}(\tilde{v})-\tilde{v}\|_{L^2(\Omega)}.
\end{align*}
Hence, by orthogonality we have that for $f:= R_{\alpha} \tilde{v} $ 
\begin{align*}
&\|R_{\alpha} \tilde{v}\|_{H^s_{\overline{W}}} \leq \frac{1}{\alpha} \|\overline{v}\|_{L^2(\Omega)}, \\
& \|A R_{\alpha}(\tilde{v})+ \sum\limits_{l=1}^{m}\tilde{\beta}_l z_l -\overline{v}\|_{L^2(\Omega)}
= \|A R_{\alpha}(\tilde{v})-\tilde{v}\|_{L^2(\Omega)}
 \leq \|\overline{v}\|_{H^{s}_{\overline{\Omega}}}\left(\log\left(1/\alpha\right)\right)^{-\mu} .
\end{align*}
Optimizing in $\alpha$ concludes the proof.
\end{proof}

\begin{proof}[Proof of Theorem \ref{thm:quant_Runge_frac}]
By combining the observations from Lemma \ref{lem:control} with the unique continuation properties from \cite{RS17} (which hold as soon as one has a pair $(w,v)$ solving \eqref{eq:dual}; in particular, it does not require a spectral assumption), we in particular infer the quantitative Runge property. 
\end{proof}

\subsection{Proofs of Theorem \ref{thm:Lip_frac_Runge_eigen} and of Proposition \ref{prop:finite_Cauchy}}
\label{sec:Proof_zero}

The proofs of Theorem \ref{thm:Lip_frac_Runge_eigen} and of Proposition \ref{prop:finite_Cauchy} follow similarly as the proof of Theorem \ref{thm:Lip_frac_Runge}. However, instead of using Alessandrini's identity as formulated in Lemma \ref{lem:Aless}, we rely on the following modification of it:

\begin{lem}
\label{lem:Aless_1}
Let $s\in (0,1)$, $n \geq 1$ and assume that $\Omega \subset \R^n$ is an open, bounded domain. Let further $W \subset \Omega_e$ be open. 
Let $u_1, u_2 \in H^{s}(\R^n)$ be solutions to \eqref{eq:frac_Cal} with potentials $q_1, q_2 \in L^{\infty}(\Omega)$. Then for any further solution $v_2$ to \eqref{eq:frac_Cal} with potential $q_2$ we have
\begin{align}
\label{eq:Aless_1}
((q_1-q_2)u_1, u_2)_{\Omega} = ((-\D)^s u_2, u_1 - v_2)_{W} - ((-\D)^s u_1 - (-\D)^s v_2, u_2)_{W}.
\end{align}
In particular,
\begin{align}
\label{eq:Cauchy}
\left| ((q_1 - q_2)u_1, u_2)_{\Omega} \right|
\leq d(\mathcal{C}_1, \mathcal{C}_2)\|u_2\|_H \|u_1\|_H,
\end{align}
with $\|w\|_H:= \|w\|_{\widetilde{H}^s(W)} + \|(-\D)^s w\|_{H^{-s}(W)}$.
\end{lem}

\begin{proof}
The generalized Alessandrini's identity is a consequence of the same ideas as the proof of Lemma \ref{lem:Aless}. Indeed, by using the equations for $u_1, u_2$, we infer
\begin{align*}
((q_1-q_2)u_1, u_2)_{\Omega} = ((-\D)^s u_2, u_1 )_{W} - ((-\D)^s u_1, u_2)_{W}.
\end{align*}
Next we note that if $v_2$ is also a solution of the equation \eqref{eq:frac_Cal} with potential $q_2$, then
\begin{align*}
0 = ((-\D)^s u_2, v_2 )_{W} - ((-\D)^s v_2, u_2)_{W}.
\end{align*}
Subtracting the previous two identities implies the claim of \eqref{eq:Aless_1}.

The estimate \eqref{eq:Cauchy} follows from \eqref{eq:Aless_1} by the definition of the aperture between two metric spaces. More precisely, we have
\begin{align*}
|((q_1-q_2)u_1, u_2)_{\Omega}|
&\leq |((-\D)^s u_2, (u_1-u_2))_{W}| + |((-\D)^s u_1 - (-\D)^s v_2, u_2)_{W}|\\
& \leq \|(-\D)^s u_2\|_{H^{-s}(W)} \|u_1 - v_2 \|_{\widetilde{H}^s(W)} + \|u_2\|_{\widetilde{H}^s(W)}\|(-\D)^s (u_1 -v_2)\|_{H^{-s}(W)}\\
&\leq (\|(-\D)^s u_2\|_{H^{-s}(W)} + \|u_2\|_{\widetilde{H}^s(W)})\|u_1 - v_2\|_{H}.
\end{align*}
Moreover, since the above inequality holds for any solution $v_2$ to \eqref{eq:frac_Cal} with potential $q_2$ it follows that 
\begin{align*}
|((q_1-q_2)u_1, u_2)_{\Omega}|\leq \left (\inf_{v_2\in \mathcal{C}_2}\frac{ \|u_1-v_2\|_{H}}{\| u_1\|_{H}} \right) \|u_1\|_H \|u_2\|_H \leq d(\mathcal{C}_1, \mathcal{C}_2) \|u_1\|_H \|u_2\|_H.
\end{align*}
\end{proof}

\begin{rmk}
\label{rmk:Cauchy_data}
If it is known that $(u_1, (-\D)^s u_1) \in \mathcal{C}_1(p^{(1)})$ and $ (u_2, (-\D)^s u_2)$, $(v_2, (-\D)^s v_2) \in \mathcal{C}_2(p^{(2)})$ for some functions $p^{(1)}, p^{(2)}$ as in the definition of the finite Cauchy data, then it is possible to replace the term $d(\mathcal{C}_1, \mathcal{C}_2)$ in \eqref{eq:Cauchy} by the quantity $d(\mathcal{C}_1(p^{(1)}), \mathcal{C}_2(p^{(2)}))$.
\end{rmk}

With Lemma \ref{lem:Aless_1} at hand, the proofs of Theorem \ref{thm:Lip_frac_Runge_eigen} and Proposition \ref{prop:finite_Cauchy} follow analogously as the proof of Theorem \ref{thm:Lip_frac_Runge}: It suffices to replace the use of the Runge approximation result from \cite{RS17}, which was recalled in Theorem \ref{thm:Runge_frac}, by the Runge approximation from Theorem \ref{thm:quant_Runge_frac}. We discuss some of the details of this. As there are no major differences in the arguments for the proofs of Proposition \ref{prop:finite_Cauchy} and for Theorem \ref{thm:Lip_frac_Runge_eigen}, we only present the details for the proof of Proposition \ref{prop:finite_Cauchy}. 

\begin{proof}[Proof of Proposition \ref{prop:finite_Cauchy}]
As in the proof of Theorem \ref{thm:Lip_frac_Runge} we choose functions $h_{1}^{(k)}, \dots, h_{m}^{(k)}$, $k\in \{1,2\}$, such that the matrix
\begin{align*}
M:= \begin{pmatrix}
(g_1, h_{1}^{(1)} h_{1}^{(2)})_{\Omega} & \dots & (g_m, h_{1}^{(1)} h_{1}^{(2)})_{\Omega}\\ 
\vdots & \vdots & \vdots\\
(g_1, h_{m}^{(1)} h_{m}^{(2)})_{\Omega} & \dots & (g_m, h_{m}^{(1)} h_{m}^{(2)})_{\Omega} 
\end{pmatrix}
\end{align*}
is invertible (by choosing $h_{l}^{(k)}$ suitably, we can even ensure that the matrix $M$ is arbitrarily close to the identity matrix). By the Runge approximation result of Theorem \ref{thm:quant_Runge_frac}, we obtain solutions $u_{l}^{(k)} \in H_2$ to the Schrödinger equation \eqref{eq:frac_Cal} and kernel elements $z_{l}^{(k)} \in Z_2$ such that
\begin{align*}
h_{l}^{(k)} = u_{l}^{(k)} + z_{l}^{(k)} + r_{l}^{(k)} \mbox{ in } \Omega, \ u_l^{(k)} = f_l^{(k)} \mbox{ in } \Omega_e,
\end{align*}
where $k\in \{1,2\}$, and
\begin{align*}
\|r_{l}^{(k)}\|_{L^2(\Omega)} \leq \epsilon \|h_{l}^{(k)}\|_{H^{s}_{\overline{\Omega}}}, \ 
\|u_{l}^{(k)}\|_{H^{s}_{\overline{W}}} \leq C e^{C \epsilon^{-\mu}} \|h_{l}^{(k)}\|_{L^2(\Omega)}.
\end{align*}
Relying on the estimates from Lemma \ref{lem:Aless_1}, we obtain
\begin{align}
\label{eq:eigen_bda}
\begin{split}
\left|  ((q_1 - q_2) h_l^{(1)}, h_l^{(2)}) \right|
&\leq 
\left|  ((q_1 - q_2) (u_l^{(1)} + z_l^{(1)}), (u_l^{(2)} + z_l^{(2)})) \right|\\
& \quad + |((q_1 - q_2) r_j^{(1)}, h_j^{(2)})| + |((q_1 - q_2) r_j^{(2)}, h_j^{(1)})| + |((q_1 - q_2) r_j^{(1)}, r_j^{(2)})|\\
&\leq  d(\mathcal{C}_1(f^{(1)}), \mathcal{C}_2(f^{(2)})) \|u_j^{(1)} + z_j^{(1)}\|_H \|u_j^{(2)} + z_j^{(2)}\|_H\\
& \quad + |((q_1 - q_2) r_j^{(1)}, h_j^{(2)})| + |((q_1 - q_2) r_j^{(2)}, h_j^{(1)})| + |((q_1 - q_2) r_j^{(1)}, r_j^{(2)})|\\
& \leq  d(\mathcal{C}_1(f^{(1)}), \mathcal{C}_2(f^{(2)})) \|u_j^{(1)} + z_j^{(1)}\|_H \|u_j^{(2)} + z_j^{(2)}\|_H \\
& \quad + \epsilon \|q_1 - q_2\|_{L^{\infty}(\Omega)} (\|h_j^{(1)}\|_{H^s_{\overline{\Omega}}} + \|h_j^{(2)}\|_{H^s_{\overline{\Omega}}}).
\end{split}
\end{align}
Next we estimate the term $\|u_j^{(1)} + z_j^{(1)}\|_H $. 
First we note that
\begin{align}
\label{eq:u1}
\begin{split}
\|u_j^{(1)} + z_j^{(1)}\|_{H} 
&= \|u_j^{(1)} + z_j^{(1)}\|_{\widetilde{H}^s(W)} + \|(-\D)^s (u_j^{(1)}+z_j^{(1)})\|_{H^{-s}(W)}\\
&= \|f_j^{(1)}\|_{\widetilde{H}^s(W)} +  \|(-\D)^s (u_j^{(1)}+z_j^{(1)})\|_{H^{-s}(W)},
\end{split}
\end{align}
where we have used that $u_j^{(1)}+z_j^{(1)}= f_j^{(1)}$ in $W$. Next, we have
\begin{align}
\label{eq:u2}
\begin{split}
\|(-\D)^s (u_j^{(1)} + z_j^{(1)})\|_{H^{-s}(W)} 
&\leq \|(-\D)^s (u_j^{(1)} + z_j^{(1)}-f_j^{(1)})\|_{H^{-s}(W)} + \|(-\D)^s f_j^{(1)}\|_{H^{-s}(W)} \\
&\leq \|(-\D)^s (u_j^{(1)}+ z_j^{(1)}-f_j^{(1)})\|_{H^{-s}(W)} + C \|f_j^{(1)}\|_{\widetilde{H}^{s}(W)}.
\end{split}
\end{align}
We have used $f_j^{(1)} \in H^{s}_{\overline{W}}$ in order to deal with the second contribution on the right hand side.
Since $z_j^{(1)}, u_j^{(1)}+ z_j^{(1)}-f_j^{(1)} \in H^{s}_{\overline{\Omega}}$, we infer for $x\in W$ that
\begin{align}
\label{eq:u3}
\begin{split}
|(-\D)^s (u_j^{(1)}+ z_j^{(1)}-f_j^{(1)})(x) |
&= \left| c_{s,n}\int\limits_{\Omega} \frac{(u_j^{(1)}+ z_j^{(1)}-f_j^{(1)})(y)}{|x-y|^{n+2s}} dy \right|\\
&\leq c_{n,s} |\Omega|^{1/2}(|x|+ C_{\Omega, W})^{-n-2s} \|(u_j^{(1)}+ z_j^{(1)}-f_j^{(1)})\|_{L^2(\Omega)}\\
&= c_{n,s} |\Omega|^{1/2}(|x|+ C_{\Omega, W})^{-n-2s} \|u_j^{(1)}+ z_j^{(1)}\|_{L^2(\Omega)}.
\end{split}
\end{align}
Hence, integrating this over $x\in W$ we obtain
\begin{align}
\label{eq:u4}
\|(-\D)^s (u_j^{(1)}+ z_j^{(1)}-f_j^{(1)})\|_{H^{-s}(W)} \leq
\|(-\D)^s (u_j^{(1)}+ z_j^{(1)}-f_j^{(1)})\|_{L^{2}(W)}
\leq C \|u_j^{(1)}+ z_j^{(1)}\|_{L^2(\Omega)}.
\end{align}
In order to bound the right hand side of \eqref{eq:u4}, we use the decomposition $h^{(1)}_j = u_j^{(1)} + z_j^{(1)} + r_j^{(1)}$ and the $L^2(\Omega)$ orthogonality of $u_j^{(1)}$ and $z_j^{(1)}$. We obtain
\begin{align}
\label{eq:z1}
 \|z_j^{(1)} + u_j^{(1)}\|_{L^2(\Omega)} \leq \|h_j^{(1)}\|_{L^2(\Omega)} + \|r_j^{(1)}\|_{L^2(\Omega)}
\leq 2 \|h_j^{(1)}\|_{L^2(\Omega)}.
\end{align}
Here we used that we can bound $ \|r_j^{(1)}\|_{L^2(\Omega)} \leq \epsilon \|h_j^{(1)}\|_{H^{s}_{\overline{\Omega}}} \leq \|h_j^{(1)}\|_{L^2(\Omega)}$, if $\epsilon \leq \frac{\|h_j^{(1)}\|_{L^2(\Omega)}}{\|h_j^{(1)}\|_{H^{s}_{\overline{\Omega}}}}$ (c.f. Lemma \ref{lem:control}).
Summing up the estimates \eqref{eq:u1}-\eqref{eq:z1} and using that $\|f_j^{(1)}\|_{\widetilde{H}^{s}(W)}\leq C e^{C \epsilon^{-\mu}} \|h_j^{(1)}\|_{L^2(\Omega)}$, we thus arrive at
\begin{align}
\label{eq:combi}
\|u^{(1)}_j+ z_j^{(1)}\|_H \leq (C e^{C \epsilon^{-\mu}} + 2) \|h_j^{(1)}\|_{L^2(\Omega)}).
\end{align}

Assuming that
\begin{align*}
q_1 - q_2 = \sum\limits_{j=1}^{m} a_j g_j,
\end{align*}
and using the estimate \eqref{eq:Cauchy} from Lemma \ref{lem:Aless_1} and Remark \ref{rmk:Cauchy_data} in combination with \eqref{eq:combi} and \eqref{eq:eigen_bda}, we infer
\begin{align*}
\left| \sum\limits_{j=1}^{m} a_j (g_j, h_{l}^{(1)} h_{l}^{(2)})_{\Omega} \right|
&\leq C d(\mathcal{C}_1(f^{(1)}), \mathcal{C}_2(f^{(2)})) e^{C \epsilon^{-\mu}}\|h_l^{(1)}\|_{H^s_{\overline{\Omega}}}\|h_l^{(2)}\|_{H^s_{\overline{\Omega}}}\\
& \quad + \epsilon \|q_1 - q_2\|_{L^{\infty}(\Omega)} \|h_l^{(1)}\|_{H^s_{\overline{\Omega}}}\|h_l^{(2)}\|_{H^s_{\overline{\Omega}}}.
\end{align*}
The proof of Proposition \ref{prop:finite_Cauchy} is now concluded by the same absorption argument as in the proof of Theorem \ref{thm:Lip_frac_Runge}.
\end{proof}

\section{Necessity of an Exponential Dependence in the Case of Piecewise Constant Potentials}
\label{sec:opti}

In Example \ref{ex:pc} we dealt with the setting of piecewise constant potentials associated with a partitioning of $\Omega$ into $N$ Lipschitz subdomains $D_1, \dots, D_N$. In this context we saw that by precisely keeping track of the dependences in the proof of Theorem \ref{thm:Lip_frac_Runge}, the constant $C_1>1$ from Theorem \ref{thm:Lip_frac_Runge} could be bounded by $e^{C N^{\tilde{\mu}}}$ for some $\tilde{\mu}>0$. In this section we show that, as in the case of the classical Calder\'on problem (c.f. \cite{Ron06}), (up to the precise choice of the exponent) such an exponential dependence on the constant $N>1$ is also necessary in general. 
In order to observe this, we combine the strategy introduced by Rondi \cite{Ron06}, which is based on an adaptation of the instability argument of Mandache \cite{M01} with the constructions given in \cite{RS17d}.

\begin{thm}
\label{thm:exp_dep}
Let $s\in (0,1)$, let $\Omega =B_1(0)$ be the unit ball in $\R^n$ for some $n \in \N$ and let $W:= B_3(0) \setminus \overline{B_2(0)}$. Let $\lambda_1$ denote the first Dirichlet eigenvalue of $(-\D)^s$ on $\Omega$. Then there exist $N_0 \in \N$, $\mu>0$ and $\bar{C}\geq 1$ such that for all $N\geq N_0$ there is a partitioning $D_1, \dots, D_N$ of $\Omega$ into Lipschitz sets,
such that the Lipschitz constant $C_1$ from Theorem \ref{thm:Lip_frac_Runge} associated with the choice 
\begin{align*}
q_1, q_2 \in \spa \{\chi_{D_1}, \dots, \chi_{D_N}\}, \ \|q_1\|_{L^{\infty}(\Omega)}, \|q_2\|_{L^{\infty}(\Omega)} \leq \frac{ \lambda_1}{2}
\end{align*}
necessarily satisfies
\begin{align*}
C_1 \geq \exp(\bar{C} N^{\mu}).
\end{align*}
\end{thm}

\begin{proof}
Let $\tilde{C}_1>1$ be the Lipschitz constant in the estimate of Theorem \ref{thm:Lip_frac_Runge} for the knowledge of the \emph{full}, infinite measurement Dirichlet-to-Neumann map with potentials $q_1, q_2$ and a set $\{g_1, \dots, g_N\}$ as in Theorem \ref{thm:exp_dep}, i.e. assume that for all $q_1, q_2 \in \spa \{\chi_{D_1}, \dots, \chi_{D_N}\}$ it holds
\begin{align}
\label{eq:full_est}
\|q_1 - q_2\|_{L^{\infty}(\Omega)} \leq \tilde{C}_1 \|\Lambda_{q_1}-\Lambda_{q_2}\|_{\ast},
\end{align}
where $\|\Lambda_{q_1}-\Lambda_{q_2}\|_{\ast}:= \sup\{((\Lambda_{q_1}-\Lambda_{q_2})f_1, f_2)_{W}: \|f_1\|_{\widetilde{H}^{s}(W)}=1=\|f_2\|_{\widetilde{H}^s(W)} \}$.
Since $\|\Lambda_{q_1}-\Lambda_{q_2}\|_{\ast} \geq \|(\Lambda_{q_1}-\Lambda_{q_2})\tilde{f}_1\|_{H^{-s}(W)}$ for $\tilde{f_1} \in \widetilde{H}^s(W)$ with $\|\tilde{f}_1\|_{\widetilde{H}^{s}(W)}=1$, we note that $C_1 \geq \tilde{C}_1$. Hence it suffices to prove the claimed lower bound for the constant $\tilde{C}_1$.

In order to prove this, we follow the strategy of Rondi \cite{Ron06}. Thus, on the one hand we construct an exponentially large $\delta$-discrete set $X \subset \mathcal{Q}_N$, where
\begin{align*}
\mathcal{Q}_N:=\left\{q\in L^{\infty}(\Omega): \ q \mbox{ is piecewise constant}, \ \|q\|_{L^{\infty}(\Omega)}\leq \frac{\lambda_1}{2}  \right\}, \ d_{\mathcal{Q}_N}(q_1, q_2)= \|q_1-q_2\|_{L^{\infty}(\Omega)}.
\end{align*}
We recall that a set is said to be $\delta$-discrete, if all its elements are at a mutual distance larger or equal to $\delta$.
On the other hand, we show that the set of Dirichlet-to-Neumann maps 
\begin{align*}
\mathcal{L}:=\left\{\Lambda_{q}: q \in \mathcal{Q}_N, \ \|q\|_{L^{\infty}(\Omega)} \leq \frac{\lambda_1}{2} \right\}
\end{align*}
equipped with the operator norm $d_{\mathcal{L}}(\Lambda_{q_1},\Lambda_{q_2}):= \|\Lambda_{q_1}-\Lambda_{q_2}\|_{\widetilde{H}^{s}(W) \rightarrow H^{-s}(W)}$ has a relatively small $\epsilon$-net, i.e. there exists a subset of the set of $\mathcal{L}$ such that all points in $\mathcal{L}$ are at most $\epsilon$ far from this subset. Combining these two observations with the estimate from Theorem \ref{thm:Lip_frac_Runge} then yields the claimed lower bound on the constant $\tilde{C}_1>0$ (and hence also on the constant $C_1>0$ from Theorem \ref{thm:Lip_frac_Runge}).

Let us be more precise. We first cover $\Omega$ by (up to null-sets) disjoint cubes of side length $N^{-1}$ and a remainder (e.g. by covering $\Omega$ by a suitable coordinate grid of the desired lengths scales). We consider the first $N$ of these cubes. This allows us to construct a large $\delta$-discrete set on $\Omega$ for an arbitrary choice of $\delta\in (0,\min\{\lambda_1/2,1\})$: Indeed, by considering the piecewise constant functions, which on the chosen $N$ cubes attain any of the values $\pm \delta, 0$ and vanish outside of the $N$ cubes, we obtain a $\delta$-discrete set consisting of $3^N$ elements.

Using the strong smoothing properties of the Dirichlet-to-Neumann map for the fractional Laplacian, which were studied in \cite{RS17d}, it is possible to construct a small $\epsilon$-net for $\mathcal{L}$. Indeed, an orthonormal basis with smoothing properties in our geometry was constructed in Lemma 2.1 in \cite{RS17d}. Based on that, for any $\epsilon \in (0,1)$ an $\epsilon$-net of $\mathcal{L}$ with respect to the distance $
\tilde{d}_{\mathcal{L}}(\Lambda_{q_1},\Lambda_{q_2}):= \|
\Lambda_{q_1}-\Lambda_{q_2}\|_{L^2(W)\rightarrow L^2(W)}$ was 
constructed in Lemma 3.2 in \cite{RS17d}. Its cardinality was estimated by $
\exp(\eta(-\log(\epsilon))^{2n-1})$.

We seek to exploit these two observations. To this end, we first relate the $L^2$ and $H^{s}$ based operator norms of the fractional Dirichlet-to-Neumann maps. We infer that
\begin{align*}
\sup\limits_{\|f_1\|_{\widetilde{H}^{s}(W)}=1=\|f_2\|_{\widetilde{H}^s(W)}} \int\limits_{W}((\Lambda_{q_1}-\Lambda_{q_2}) f_1, f_2) dx 
&\leq \|\Lambda_{q_1}-\Lambda_{q_2}\|_{L^2(W) \rightarrow L^2(W)} \|f_1\|_{L^2(W)}\|f_2\|_{L^2(W)}\\
&\leq C_p\|\Lambda_{q_1}-\Lambda_{q_2}\|_{L^2(W)\rightarrow L^2(W)} ,
\end{align*}
where in the last line we used Poincar\'e's inequality $\|f_1\|_{L^2(W)} \leq C_p \|f_1\|_{\widetilde{H}^s(W)}$. Hence, in particular, $d_{\mathcal{L}}(f_1, f_2) \leq C_p \tilde{d}_{\mathcal{L}}(f_1,f_2)$. 

Next we observe that for each $\epsilon>0$ there exists $N_0=N_0(\epsilon)$ such that for all $N\geq N_0>0$ we have $3^N \geq \exp(\eta(-\log(\epsilon))^{2n-1})$, where $\eta>0$ is a constant which is independent of $\epsilon$. 

We now combine all these auxiliary results: For each $\delta \in (0, \min\{\lambda_1/2,1\})$ and $\epsilon \in (0,1) $, there exist $q_1, q_2 \in \mathcal{Q}_N$ which are $\delta$-separated with respect to $d_{\mathcal{Q}_N}(\cdot, \cdot)$ and such that $\Lambda_{q_1}, \Lambda_{q_2}$ are at distance with respect to $\tilde{d}_{\mathcal{L}}(\cdot, \cdot)$ at most $2\epsilon$ apart, such that for some constant $C>0$ (depending on $\eta$, but independent of $N$)
\begin{align}
\label{eq:final}
\delta \leq d_{\mathcal{Q}_N}(q_1,q_2) 
\leq C_1 d_{\mathcal{L}}(\Lambda_{q_1},\Lambda_{q_2}) 
\leq C_1 C_p \tilde{d}_{\mathcal{L}}(\Lambda_{q_1},\Lambda_{q_2}) 
\leq 2 C_1 C_p \epsilon
\leq 2 C_1 \exp(-C N^{1/(2n-1)}).
\end{align}
Here we have first made use of the fact that the functions $q_1, q_2$ are $\delta$-separated, then we used our Lipschitz bound, the estimate $d_{\mathcal{L}} \leq C_p \tilde{d}_{\mathcal{L}}$ and finally our choice of $N$ and $\epsilon$.
Rearranging \eqref{eq:final}, choosing for instance $\delta=\min\{\lambda_1/4, 1/2\}$ and noting that $C_p>1$ does not depend on $N$ then implies the claim.
\end{proof}

\section{\normalsize{Acknowledgments}}
The research carried out by  E. Sincich for the preparation of this paper has been supported by FRA 2016 {\it{Problemi inversi, dalla stabilit\`{a} alla ricostruzione}} funded by Universit\`{a} degli Studi di Trieste. E.S  acknowledges the support of Gruppo Nazionale per l'Analisi Matematica, la Probabilit\`{a} e le loro Applicazioni (GNAMPA) by the grant {\it{Analisi di problemi inversi:  stabilit\`{a} e
ricostruzione} }. 

\bibliographystyle{plain}
\bibliography{citations_Lipschitz_updated}

\begin{thebibliography}{10}

\bibitem{AS18}
Giovanni~S. Alberti and Matteo Santacesaria.
\newblock Calder{\'o}n's inverse problem with a finite number of measurements.
\newblock {\em arXiv preprint arXiv:1803.04224}, 2018.

\bibitem{AdHGS16}
Giovanni Alessandrini, Maarten~V. de~Hoop, Romina Gaburro, and Eva Sincich.
\newblock Lipschitz stability for the electrostatic inverse boundary value
  problem with piecewise linear conductivities.
\newblock {\em Journal de Math{\'e}matiques Pures et Appliqu{\'e}es},
  107:638--664, 2016.

\bibitem{AdHGS17}
Giovanni Alessandrini, Maarten~V. de~Hoop, Romina Gaburro, and Eva Sincich.
\newblock Lipschitz stability for a piecewise linear {S}chr{\"o}dinger
  potential from local {C}auchy data.
\newblock {\em to appear on Asymptotic Analysis}, 2017.

\bibitem{ARV09}
Giovanni Alessandrini, Luca Rondi, Edi Rosset, and Sergio Vessella.
\newblock The stability for the {C}auchy problem for elliptic equations.
\newblock {\em Inverse Problems}, 25(12):123004, 47, 2009.

\bibitem{AV05}
Giovanni Alessandrini and Sergio Vessella.
\newblock Lipschitz stability for the inverse conductivity problem.
\newblock {\em Advances in Applied Mathematics}, 35(2):207--241, 2005.

\bibitem{BdHQ13}
Elena Beretta, Maarten~V. de~Hoop, and Lingyun Qiu.
\newblock Lipschitz stability of an inverse boundary value problem for a
  {S}chr{\"o}dinger-type equation.
\newblock {\em SIAM Journal on Mathematical Analysis}, 45(2):679--699, 2013.

\bibitem{BF11}
Elena Beretta and Elisa Francini.
\newblock Lipschitz stability for the electrical impedance tomography problem:
  the complex case.
\newblock {\em Comm. Partial Differential Equations}, 36:1723 -- 1749, 2011.

\bibitem{CL18}
Xinlin Cao and Hongyu Liu.
\newblock Determining a fractional {H}elmholtz system with unknown source and
  medium parameter.
\newblock {\em arXiv preprint arXiv:1803.09538}, 2018.

\bibitem{DSV14}
Serena Dipierro, Ovidiu Savin, and Enrico Valdinoci.
\newblock All functions are locally {$s$}-harmonic up to a small error.
\newblock {\em J. Eur. Math. Soc. (JEMS)}, 19(4):957--966, 2017.

\bibitem{FF14}
Mouhamed~Moustapha Fall and Veronica Felli.
\newblock Unique continuation property and local asymptotics of solutions to
  fractional elliptic equations.
\newblock {\em Communications in Partial Differential Equations},
  39(2):354--397, 2014.

\bibitem{GS15}
Romina Gaburro and Eva Sincich.
\newblock Lipschitz stability for the inverse conductivity problem for a
  conformal class of anisotropic conductivities.
\newblock {\em Inverse Problems}, 31(1):015008, 2015.

\bibitem{G08}
Bastian Gebauer.
\newblock Localized potentials in electrical impedance tomography.
\newblock {\em Inverse Probl. Imaging}, 2(2):251--269, 2008.

\bibitem{GLX17}
Tuhin Ghosh, Yi-Hsuan Lin, and Jingni Xiao.
\newblock The {C}alder{\'o}n problem for variable coefficients nonlocal
  elliptic operators.
\newblock {\em Communications in Partial Differential Equations},
  42(12):1923--1961, 2017.

\bibitem{GRSU18}
Tuhin Ghosh, Angkana R{\"u}land, Mikko Salo, and Gunther Uhlmann.
\newblock Uniqueness and reconstruction for the fractional {C}alder{\'o}n
  problem with a single measurement.
\newblock {\em arXiv preprint arXiv:1801.04449}, 2018.

\bibitem{GSU16}
Tuhin Ghosh, Mikko Salo, and Gunther Uhlmann.
\newblock The {C}alder{\'o}n problem for the fractional {S}chr{\"o}dinger
  equation.
\newblock {\em arXiv preprint arXiv:1609.09248}, 2016.

\bibitem{Gr15}
Gerd Grubb.
\newblock Fractional {L}aplacians on domains, a development of
  {H}{\"o}rmander's theory of {$\mu$}-transmission pseudodifferential
  operators.
\newblock {\em Advances in Mathematics}, 268:478--528, 2015.

\bibitem{HL17}
Bastian Harrach and Yi-Hsuan Lin.
\newblock Monotonicity-based inversion of the fractional {S}chr{\"o}dinger
  equation.
\newblock {\em arXiv preprint arXiv:1711.05641}, 2017.

\bibitem{HPS17}
Bastian Harrach, Valter Pohjola, and Mikko Salo.
\newblock Monotonicity and local uniqueness for the {H}elmholtz equation.
\newblock {\em arXiv preprint arXiv:1709.08756}, 2017.

\bibitem{LL17}
Ru-Yu Lai and Yi-Hsuan Lin.
\newblock Global uniqueness for the semilinear fractional {S}chr{\"o}dinger
  equation.
\newblock {\em arXiv preprint arXiv:1710.07404}, 2017.

\bibitem{M01}
Niculae Mandache.
\newblock Exponential instability in an inverse problem for the
  {S}chr{\"o}dinger equation.
\newblock {\em Inverse Problems}, 17(5):1435, 2001.

\bibitem{McLean}
William McLean.
\newblock {\em Strongly elliptic systems and boundary integral equations}.
\newblock Cambridge University Press, Cambridge, 2000.

\bibitem{Ron06}
Luca Rondi.
\newblock A remark on a paper by {A}lessandrini and {V}essella.
\newblock {\em Advances in Applied Mathematics}, 36(1):67--69, 2006.

\bibitem{Rue15}
Angkana R{\"u}land.
\newblock Unique continuation for fractional {S}chr{\"o}dinger equations with
  rough potentials.
\newblock {\em Communications in Partial Differential Equations},
  40(1):77--114, 2015.

\bibitem{RS17}
Angkana R{\"u}land and Mikko Salo.
\newblock The fractional {C}alder{\'o}n problem: Low regularity and stability.
\newblock {\em ArXiv preprint, August}, 2017.

\bibitem{RS17d}
Angkana R{\"u}land and Mikko Salo.
\newblock Exponential instability in the fractional {C}alder{\'o}n problem.
\newblock {\em Inverse Problems}, 34(4):045003, 2018.

\bibitem{S17}
Mikko Salo.
\newblock The fractional {C}alder{\'o}n problem.
\newblock {\em arXiv preprint arXiv:1711.06103}, 2017.

\bibitem{Yu16}
Hui Yu.
\newblock Unique continuation for fractional orders of elliptic equations.
\newblock {\em arXiv preprint arXiv:1609.01376}, 2016.

\end{thebibliography}

\end{document}